\documentclass[12pt]{amsart}
\usepackage{amsmath, amssymb, amsthm, latexsym}
\usepackage{amssymb,amscd,amsmath}
\usepackage{latexsym}
\usepackage[center]{caption}
\usepackage{tikz}
\newcounter{braid}
\newcounter{strands}

\DeclareMathAlphabet{\bsf}{OT1}{cmss}{bx}{n}

\pgfkeyssetvalue{/tikz/braid height}{1cm}
\pgfkeyssetvalue{/tikz/braid width}{1cm}
\pgfkeyssetvalue{/tikz/braid start}{(0,0)}
\pgfkeyssetvalue{/tikz/braid colour}{black}
\pgfkeys{/tikz/strands/.code={\setcounter{strands}{#1}}}

\makeatletter
\def\cross{%
  \@ifnextchar^{\message{Got sup}\cross@sup}{\cross@sub}}

\def\cross@sup^#1_#2{\render@cross{#2}{#1}}

\def\cross@sub_#1{\@ifnextchar^{\cross@@sub{#1}}{\render@cross{#1}{1}}}

\def\cross@@sub#1^#2{\render@cross{#1}{#2}}

\def\render@cross#1#2{
  \def\strand{#1}
  \def\crossing{#2}
  \pgfmathsetmacro{\cross@y}{-\value{braid}*\braid@h}
  \pgfmathtruncatemacro{\nextstrand}{#1+1}
  \foreach \thread in {1,...,\value{strands}}
  {
    \pgfmathsetmacro{\strand@x}{\thread * \braid@w}
    \ifnum\thread=\strand
    \pgfmathsetmacro{\over@x}{\strand * \braid@w + .5*(1 - \crossing) * \braid@w}
    \pgfmathsetmacro{\under@x}{\strand * \braid@w + .5*(1 + \crossing) * \braid@w}
    \draw[braid] \pgfkeysvalueof{/tikz/braid start} +(\under@x pt,\cross@y pt) to[out=-90,in=90] +(\over@x pt,\cross@y pt -\braid@h);
    \draw[braid] \pgfkeysvalueof{/tikz/braid start} +(\over@x pt,\cross@y pt) to[out=-90,in=90] +(\under@x pt,\cross@y pt -\braid@h);
    \else
    \ifnum\thread=\nextstrand
    \else
     \draw[braid] \pgfkeysvalueof{/tikz/braid start} ++(\strand@x pt,\cross@y pt) -- ++(0,-\braid@h);
    \fi
   \fi
  }
  \stepcounter{braid}
}

\tikzset{braid/.style={double=\pgfkeysvalueof{/tikz/braid colour},double distance=1pt,line width=2pt,white}}

\newcommand{\braid}[2][]{%
  \begingroup
  \pgfkeys{/tikz/strands=2}
  \tikzset{#1}
  \pgfkeysgetvalue{/tikz/braid width}{\braid@w}
  \pgfkeysgetvalue{/tikz/braid height}{\braid@h}
  \setcounter{braid}{0}
  \let\sigma=\cross
  #2
  \endgroup
}
\makeatother

\input xypic
\newtheorem{theorem}{Theorem}

\newtheorem{lemma}[theorem]{Lemma}

\newtheorem{corollary}[theorem]{Corollary}


\def\Z{\mathbb{Z}}

\def\calA{\mathcal{A}}

\def\Pi{\mathbb{P}^{\infty}}

\def\md{\mathcal{D}}

\def\qed{\hfill$\square$\medskip}

\def\Zpk{\mathbb{Z}/p^{k}}
\def\Zpk1{\mathbb{Z}/p^{k-1}}

\newcommand{\rref}[1]{(\ref{#1})}

\newcommand{\beg}[2]{\begin{equation}\label{#1}#2\end{equation}}
\def\r{\rightarrow}

\def\sl2{\widetilde{SL_{2}(\Z)}}

\def\md
\def\rank{\operatorname{rank}}

\title[The super reciprocal plane]{Equivariant cohomology and the super reciprocal plane of a hyperplane arrangement}
\author{Sophie Kriz}

\begin{document}

\begin{abstract}
In this paper, we investigate certain graded-commuta-tive rings which are related to the reciprocal plane compactification of the coordinate ring of a complement of a hyperplane arrangement. We give a presentation of these rings by generators and defining relations. This presentation was used by Holler and I. Kriz \cite{h} to calculate the $\mathbb{Z}$-graded coefficients of localizations of ordinary $RO((\mathbb{Z}/p)^n)$-graded equivariant cohomology at a given set of representation spheres, and also more recently by the author \cite{Krizme} in a generalization to the case of an arbitrary finite group. We also give an interpretation of these rings in terms of
superschemes, which can be used to further illuminate their structure.

\end{abstract}

\maketitle

\vspace{3mm}
\section{Introduction}
$G$-equivariant generalized homology and cohomology theory for a compact lie group $G$ is best behaved when the (co)-homology groups are graded by elements of the real representation ring $RO(G)$. In this case (see Lewis, May, Steinberger
\cite{lms} for background), the theory enjoys many of the properties of non-equivariant (co)-homology, for example, Spanier-Whitehead duality. Explicit calculations of equivariant cohomology groups, however, are much harder than in the non-equivariant case. A telling example is the case of ``ordinary" $G$-equivariant cohomology theories, defined by Lewis, May and McClure \cite{lmc}. These theories satisfy a ``dimension axiom" in the sense that the $\mathbb{Z}$-graded part of their coefficients (i.e.
(co)-homology of a point) are zero except in dimension $0$ for all (closed) subgroups of $G$.

However, calculation of the $RO(G)$-graded coefficients of these ``ordinary" $G$-equivariant cohomology theories has been an open problem since the
1980s, and these groups carry some deep information. For example, for the ``constant" $\underline{\Z}$ Mackey functor coefficients, (which means that restrictions to subgroups are identities), a partial calculation of the $RO(G)$-graded coefficients for $G=\Z/8$ was a key ingredient in the solution by Hill, Hopkins and Ravenel \cite{km} of the Kervaire invariant $1$ problem.

The algebraic calculations made in the present paper are relevant to the ordinary $RO(G)$-graded (co)homology theory with constant $\underline{\Z/p}$ coefficients for $G=(\Z/p)^n$. We denote this theory by $H\underline{\Z/p}_{(\Z/p)^n}$. In the paper \cite{hk}, Holler and I. Kriz
calculated the ``positive" part of these coefficients, meaning the groups
\beg{positivecoeffs}{H\underline{\Z/p}_{(\Z/p)^n}^V(*)}
with $V$ an actual (not virtual) representation for $p=2$. A key ingredient in this calculation was the geometric fixed point ring
\beg{geomfixpts}{(\Phi^{(\Z/p)^n} H\underline{\Z/p})_*,}
which is the localization of the full $RO((\Z/p)^n)$-graded coefficient ring by inverting the inclusions $S^0\r S^\alpha$ for all non-trivial irreducible
representations $\alpha$ (see Tom Dieck \cite{tom} and \cite{lms}, chapter 11, Def. 9.7).

Holler and I. Kriz \cite{hk} calculated the ring \rref{geomfixpts} for $p=2$ by hand using a spectral sequence, and commented that the rings
seemed to have an unusual algebraic structure, and asked about its geometric significance. They also did not know how to complete the same computation for $p>2$, where the structure seemed much more complicated.

Answering the second question is the main purpose of the present paper. Using our main theorem (Theorem 2 below),
Holler and I. Kriz \cite{h} then generalized their calculations of the geometric fixed point coefficient ring
 \rref{geomfixpts} to $p>2$, and also answered the following more general question:

What is the structure of the $\Z$-graded coefficient ring $R_S$ of the $(\Z/p)^n$-fixed point spectrum given by localizing
$H\underline{\Z/p}_{(\Z/p)^n}$ by inverting the maps $S^0\r S^\alpha$ for a given set $S$ of
irreducible $(\Z/p)^n$-representations?

Symbolically, we may write
\beg{wedge1}{R_S= ((\bigwedge_{i=1}^m S^{\infty\alpha_i})\wedge H \underline{\Z/p})^{(\Z/p)^n}_*}
where $S=\{\alpha_1,\dots, \alpha_m\}$. 

Then, in particular, the geometric fixed point coefficient ring \rref{geomfixpts} is equal to $R_S$ where 
$$S=\{\alpha_1,\dots,\alpha_{p^n-1}\}$$ consists of all non-trivial irreducible
representations of $(\Z/p)^n$.

In \cite{hk} Theorem 2, Holler and I. Kriz proved that
\beg{geomfixptsz2}{\begin{array}{c}
(\Phi^{(\Z/2)^n} H\underline{\Z/2})_* = \\
\Z/2[t_\alpha|\alpha\in (\Z/2)^n\setminus\{0\}]/(t_\alpha t_\beta+t_\alpha t_\gamma + t_\beta t_\gamma | \alpha +\beta + \gamma =0),
\end{array}}
where $t_\alpha$ are in degree $1$. They proved this by counting the dimension of the submodule of homogeneous elements of a given degree and
matching it with a spectral sequence. But what do these relations mean?

Consider the affine space
$$\mathbb{A}_{\mathbb{F}_2}^n=\text{Spec} (\mathbb{F}_2[x_1,\dots,x_n]).$$
Then consider a set of elements $z_\alpha$ which are non-zero linear combinations of the coordinates $x_1,\dots, x_n$
with coefficients in $\mathbb{F}_2$. Such linear combinations can, in turn, be identified with equations of hyperplanes
through the origin in $\mathbb{A}_{\mathbb{F}_2}^n$. (In the case of \rref{geomfixptsz2}, all possible rational hyperplanes, as it turns out.) If we remove
these hyperplanes from $\mathbb{A}_{\mathbb{F}_2}^n$, we obtain an affine variety with coordinate ring
\beg{coordring1}{(\prod_{\alpha\in (\Z/2)^n\setminus \{0\}}z_\alpha^{-1})\mathbb{F}_2[x_1,\dots,x_n].}
The ring \rref{geomfixptsz2} is isomorphic to the subring of the ring \rref{coordring1} generated by the elements $t_\alpha=z_\alpha^{-1}.$ 
This result turned out to be known (for example, \cite{foot}, Theorem 4). In fact, the affine variety with coordinate ring
\rref{geomfixptsz2} is known as the {\em reciprocal plane} of the hyperplane arrangement $\{z_\alpha\}$ (see \cite{den}).

\vspace{5mm}

The main contribution of the present paper is finding an analog of this story for $p>2$. 
From the point of view of algebraic geometry, there is no difference:
As we already mentioned, the reciprocal plane construction is independent of characteristic.

In algebraic topology, however, when we are dealing with characteristic $p\neq 2$, coefficient rings become {\em graded-commutative}, i.e.
$$xy=(-1)^{|x||y|} yx$$
where $|x|$ denotes the degree of $x$. So to solve the structure of the rings \rref{geomfixpts}, \rref{wedge1} for $p>2$, it was
necessary to discover the {\em appropriate graded-commutative analogue} of the reciprocal plane, and to prove structure results analogous to 
\cite{foot}. This is the main result of the present paper. 

Very briefly, we consider the ring
$$\mathbb{F}_p[x_1,\dots,x_n]\otimes \Lambda_{\mathbb{F}_p}[dx_1,\dots, dx_n]$$
where $\Lambda$ denotes the exterior algebra.
In this ring, invert a set of linear combinations $z_\alpha$ of the elements $x_i$. The right ring turns out to be the subring generated by 
$t_\alpha=z_\alpha^{-1}$ and $u_\alpha=z_\alpha^{-1} dz_{\alpha}$. Topologically, the element $t_{\alpha}$
has degree 2 and the element $u_{\alpha}$ has degree 1, corresponding to the fact that we are dealing with complex, not real, representations
for $p>2$.

In this paper, I determine the structure of these subrings in a way analogous to (but more complicated than) the commutative case. Holler and I. Kriz \cite{h} then used my structure theorems to prove that these rings
are isomorphic to the rings \rref{wedge1} for $p>2$.
In a recent follow-up paper \cite{Krizme}, I also used these results to obtain a generalization to all finite groups. 
These are the main topological applications of the results of the present paper.

On the geometric side, the $Spec$ of a graded-commutative ring is a {\em superscheme} (for a survey, see \cite{west}).
In Section \ref{s5}, I develop the superscheme analog of some of the known geometric structures associated with the reciprocal plane, 
which correspond to my algebraic 
generalization to graded-commutative rings. Again, the algebraic geometry side of the story is independent of characteristic.

\vspace{5mm}
 
The present paper is organized as follows: In the next section, I give precise statements of the algebraic results of this paper.
In Section \ref{s2}, I give a proof of the main theorem and also prove that the relation ideal $K$
is also generated by the relation polynomials $P_{L,S}$ where the $L$'s are restricted to ``minimal" relations.
In Section \ref{s5}, I discuss the geometric interpretation, including the construction of the superscheme corresponding to the graded-commutative case (Theorem \ref{t30}).

\vspace{5mm}
\section{Statement of the results}\label{Statementoftheresults}

\vspace{3mm}
Following Terao \cite{terao}, consider an $n$-dimensional affine space 
$$\mathbb{A}_F^n= Spec (F[x_1,\dots, x_n])$$
over a field $F$.
Let $z_1,\dots,z_m$ be non-zero linear combinations of the coordinates
$x_1,\dots,x_n$ with coefficients in $F$.
We can think of the $z_i$'s as equations of hyperplanes in $\mathbb{A}_F^n$. Then the coordinates $t_i=z_i^{-1}$ define a morphism of affine 
varieties

$$\pi :\mathbb{A}_F^n\setminus Z(z_1\dots z_m)\rightarrow \mathbb{A}_F^m$$
where $ZI=Z(I)$ is the set of zeros of an ideal $I$. The morphism $\pi$ is an embedding if the $z_j$'s linearly span the $x_i$'s. Consider 
the Zariski closure of Im($\pi$). As we shall see, this variety is a cone, so we can speak of the corresponding projective variety.
This construction, called the reciprocal plane, has been studied extensively 
(see \cite{foot, loo03, ht03, ps06, st09, ssv11, hk11, len11}).
For a survey, see \cite{den}.

To understand this construction better, we must describe it algebraically, which will also bring us closer to the motivation of the present paper.
Algebraically, let
$$R=z_{1}^{-1}\dots z_{m}^{-1}F[x_1,\dots,x_{n}]=F[x_1,\dots, x_n][z_1^{-1},\dots, z_m^{-1}].$$
Then we have a homomorphism of rings
$$h : F[t_1,\dots, t_m]\rightarrow R$$
with $h(t_i)=z_i^{-1}$
(which is, of course, not onto).
Consider the ideal $I=\text{Ker}(h)$.
Denote $\mathcal{A}=\{z_1,\dots, z_m\},$ and put
$$R_{\mathcal{A},\mathbb{A}_F^n} = F[t_1,\dots,t_m]/I.$$
Then $\text{Spec} ( R_{\mathcal{A},\mathbb{A}_F^n})$ is, by definition, the Zariski closure of $\text{Im}(\pi)$. Also by the homomorphism theorem,
$R_{\mathcal{A}, \mathbb{A}_F^n}$ is a subring of $R$. Observe that $I$ is a prime ideal (therefore a radical) since $R$ is an integral domain,
and hence so are its subrings. Further, if the $z_i$'s generate the $x_j$'s, then
$$R=(t_1\cdot \dots \cdot t_m)^{-1} R_{\mathcal{A},\mathbb{A}_F^n}.$$
Thus, in particular, in this case $\pi$ is an open embedding of the hyperplane arrangement complement into the Zariski closure of its image.

The ideal $I$ is non-zero when there are linear dependencies among the hyperplane equations $z_i$.
Suppose, then, 
\beg{a1}{L=a_1z_{i_1}+\dots+a_kz_{i_k}=0\in F[x_1,\dots, x_n]} 
where $a_1, \dots,a_k\in F$ are not $0$, and 
$$1\leq i_1 <\dots<i_k\leq m.$$
So, in $R$, we have $\frac{a_1}{t_{i_1}}+\dots +\frac{a_k}{t_{i_k}}=0$ where $k>1$ (where, in the rest of this paper, we indentify 
$t_j=z_j^{-1}$). Thus, 
\beg{e1}
{ \frac{ a_1t_{i_2}\dots t_{i_k}+\dots +a_jt_{i_1}\dots\widehat{t_{i_j}  }\dots t_{i_k} 
+\dots +a_kt_{i_1}\dots t_{i_{k-1}} }{ t_{i_1}\dots t_{i_k}  }=0 \in R,} 
where the hat means an omitted term.

Hence, the numerator $P_L$ of the left hand side of (\ref{e1}) is in $I$.

\begin{theorem}\label{t1} (\cite{foot}, \cite{den}, (5.3)) Let $\mathcal{Z}$ be the set
of all nonzero linear relations $L$ among the hyperplane equations $z_i$. Then
\beg{e2}{I=(P_L(t_1,\dots,t_m)|L\in \mathcal{Z}),}
or in other words,
$$R_{\mathcal{A},\mathbb{A}_F^n}=F[t_1,\dots, t_m]/(P_L(t_1,\dots,t_m)|L\in \mathcal{Z}).$$
\label{t1}
\end{theorem}


\begin{corollary}\label{corollaryhollerIKrizGeometricExample}(\cite{h,hk})
For $p=2$, the $\Z$-graded coefficient ring \rref{a1} of the constant $\underline{\Z/2}$-Mackey functor ordinary $(\Z/2)^n$-equivariant 
cohomology spectrum with the inclusion $S^0\rightarrow S^{\alpha_i}$  inverted where $\alpha_i$ are real
irreducible representations corresponding to the hyperplanes $z_i$ is
$$R_S = R_{\mathcal{A},\mathbb{A}_{\mathbb{F}_2}^n}.$$
\end{corollary}

\vspace{3mm}

\noindent {\bf Example:} 
Formula \rref{geomfixptsz2} is a special case of Corollary \ref{corollaryhollerIKrizGeometricExample} when $\mathcal{A}$ contains all the non-zero
linear combinations of the variables $x_i$ (corresponding to all non-zero irreducible real representations of $(\Z/2)^n$).

To give a simple example of the generalization, consider $n=4$ and the hyperplanes
$$z_1 = x_1 + x_2, \;\; z_2 = x_2+x_3, \;\; z_3 = x_3+x_4,\;\; z_4 = x_1 + x_4.$$
Then the only relation among them is
$$L= z_1+z_2+z_3+z_4,$$
giving rise to
$$P_L= t_2t_3t_4+ t_1t_3t_4+t_1t_2t_4+t_1t_2t_3,$$
so we have
$$R_S = R_{\mathcal{A},\mathbb{A}_{\mathbb{F}_2}^n}= \mathbb{F}_2[t_1,t_2,t_3,t_4] / (t_2t_3t_4+ t_1t_3t_4+t_1t_2t_4+t_1t_2t_3).$$

\vspace{5mm}

For the graded-commutative case, consider
$$\Omega =F[x_1,\dots,x_n]\otimes \Lambda [dx_1,\dots,dx_n]$$
where $\Lambda$ denotes the exterior algebra over the field $F$. Then the non-zero $F$-linear combinations $z_i$ of the $x_i$'s are in the center
of $\Omega $.
Now consider 
$$T=z_1^{-1}\dots z_m^{-1}\Omega \supset \Omega.$$
This is the graded-commutative analog of the ring $R$.
We are interested in the subring $T_{\mathcal{A},\mathbb{A}_F^n}$ of $T$ generated by 
$z_1^{-1},\dots,z_m^{-1}, z_1^{-1} dz_1,\dots, z_m^{-1}dz_m$. 
Put $t_i=z_i^{-1}$and $u_i=z_i^{-1}dz_i$.
Then we have a canonical homomorphism of rings 
\beg{XiDefinitiontsus}{\psi : \Xi = F[t_1,\dots,t_m] \otimes \Lambda[u_1,\dots,u_m]\rightarrow T.} 
Let $K=\text{Ker}(\psi)$.  
By the Homomorphism Theorem, we have
$$T_{\mathcal{A},\mathbb{A}_F^n}=\Xi/K.$$


\vspace{3mm}

We want to find the generators of the ideal $K$. Recalling \rref{e2}, note that $I\subseteq K$, but in general, equality does
not arise, so we need to look for additional relations.
If $L$ is again the left hand side of \rref{a1},
then 
$$dL= a_{i_1}dz_{i_1}+\dots +a_{i_k}dz_{i_k}=0 \in T.$$
 If we multiply 
$$\begin{array}{c}
P_L=a_{i_1}t_{i_2}\dots t_{i_k}+a_{i_2}t_{i_1}\widehat {t_{i_2}}
\dots t_{i_k}+\\ \dots +a_{i_j}t_{i_1}\dots \widehat {t_{i_j}}\dots t_{i_k}+\dots +a_{i_k}t_{i_1}\dots t_{i_{k-1}} 
\end{array}$$ by
$dz_{j_1}\dots dz_{j_l}$ where 
\beg{ee2}{S=\{j_1<\dots<j_l\} \subseteq \{i_1,\dots ,i_k\},}
then some monomial summands can be expressed in terms of the $u_j$'s. 
If a monomial summand does not contain $t_{j_s}$ but does contain $dz_{j_s}$, then use $dL 
=a_{i_1}dz_{i_1}+\dots +a_{i_k}dz_{i_k}$
to eliminate $dz_{j_s}$. 
Explicitly, let
\beg{PLSDefinition}{\begin{array}{l}P_{L,S}=P_L dz_{j_1}\dots dz_{j_l}\\
-\sum _{s=1}^l t_{i_1}\dots \widehat {t_{j_s}} \dots t_{i_k} dz_{j_1}\dots \widehat {dz_{j_s}}dL\dots dz_{j_l}.\end{array}}
We have $P_{L,S} \in \Xi$. 
Note that, by definition, $P_{L,\emptyset}=P_L$.
Our main result is
\begin{theorem} \label{t30}
Let $\mathcal{Y}$ be the set of all pairs $(L,S)$ where $L$ is a linear relation among hyperplanes equations as in
\rref{a1}, and $S$ is a subset of the index set as in \rref{ee2}. Then
$$K=(P_{L,S} \; |\; (L,S)\in \mathcal{Y}).$$
In other words,
$$T_{\mathcal{A},\mathbb{A}_F^n}=\Xi /(P_{L,S} \; |\; (L,S)\in \mathcal{Y}).$$
\end{theorem}
This algebraic Theorem, along with Theorem \ref{et6t} below, was used in \cite{h} to prove the following result:
\begin{corollary} (\cite{h})
For $p>2$, the $\Z$-graded coefficient ring \rref{a1} of the constant
$\underline{\Z/p}$-Mackey functor ordinary $(\Z/p)^n$-equivariant cohomology spectrum with inclusions
$S^0\rightarrow S^{\alpha_i}$ inverted where $\alpha_i$ are complex irreducible representations corresponding to the hyperplanes $z_i$ is
$$R_S=T_{\mathcal{A},\mathbb{A}_F^n}.$$
\end{corollary}

\vspace{5mm}

\noindent{\bf Example:}
Let $L=z_1+z_2+z_3=0\in\Omega$. Then 
we have
$$P_L=P_{L,\emptyset}=\frac{z_1+z_2+z_3}{z_1z_2 z_3} = t_2t_3 +t_1t_3+ t_1 t_2.$$
Now to compute $P_{L,\{2\}}$, write
\beg{examplepldz2}{P_Ldz_2= t_2t_3 dz_2 +t_1t_2 dz_2+t_1t_3dz_2= u_2t_3+t_1u_2+t_1t_3dz_2.}
Now use
\beg{exampledL}{dL=dz_1+dz_2+dz_3=0} 
to express
$dz_2=-dz_1-dz_3,$ which we use to conclude
$$t_2t_3dz_2=-t_1t_3(dz_1+dz_3)=u_1t_3+u_3t_1.$$
Substituting this into \rref{examplepldz2} gives the relation
$$P_{L,{\{2\}}}=u_2 (t_1+t_3)-u_1t_3-u_3t_1.$$
To calculate $P_{L,\{1,2\}}$, we start with the expression
$$p_Ldz_1dz_2= t_2t_3dz_1dz_2+t_1t_2dz_1dz_2+t_1t_3dz_1dz_2=$$
$$=t_2t_3dz_1dz_2+u_1u_2+t_1t_3dz_1dz_3.$$
Using \rref{exampledL} again, we get
$$t_2t_3dz_1dz_2=t_2t_3(-dz_2-dz_3)dz_2= t_2t_3 dz_2dz_3=u_2u_3$$
and 
$$t_1t_3dz_1dz_2=t_1t_3dz_1(-dz_1-dz_3)= u_3u_1.$$
Thus, we obtain the relation
$$P_{L,{\{1,2\}}}=u_1u_2+u_2u_3 +u_3u_1.$$
The reader should keep in mind that the above derivation of examples of the relations $P_{L,S}$
is used simply to explain our definition of these relations. Nevertheless, they illustrate the fact that $P_{L,S}$ is a relation in
$t_1^{-1}\dots t_m^{-1}\Xi$ which is contained in
$\Xi$, and thus is valid in $\Xi$.

\vspace{5mm}

\section{The proof of the main result}
\label{s2}

\vspace{5mm}

In this section, we will prove Theorem \ref{t30}.
In this proof, we will use the notion of a Gr\"{o}bner basis of a module. 
Let $F$ be a field and let $R= F[x_1,\dots , x_n]$.
Consider a free module 
\beg{Freemoduleoneis}{R\{ e_1 ,\dots, e_k\}.}
By a {\em monomial} we mean a product of some powers of the $x_i$'s with one $e_j$ and possibly a non-zero coefficient from $F$.
On monomials (ignoring the coefficients) we have the TOP (term over position) lexicographic order given by 
$$x_n > \dots > x_1 > e_k> \dots > e_1.$$
A nonzero element $p$ of \rref{Freemoduleoneis} can be expressed as a sum of monomials
which are not $F$-multiples of each other, the greatest of which is called the {\em leading term} $LT(p)$.
Let $M$ be a submodule of \rref{Freemoduleoneis}. Note that by the Hilbert Basis Theorem, $M$ must be finitely generated.
A set of $R$-module generators $g_i$ of $M$ is called a {\em Gr\"{o}bner basis}
if their leading terms $LT(g_i)$ generate the submodule of \rref{Freemoduleoneis}
generated by the leading terms of all elements of $M$.
(We allow the set to include $0$, which does not affect whether it is a Gr\"{o}bner basis or not).

Then we have the following fact known as the {\em Buchberger criterion}:

\begin{theorem}\label{eisenbudthm} (Theorem 15.8, \cite{eisenbud})
Let $g_1, \dots, g_n$ be nonzero elements of $M$. Let $f_{i,j}$, $f_{j,i}$ be the monomials of minimal possible degree
such that the leading terms of $f_{j, i} \cdot g_i$ and $f_{i, j} \cdot g_j$ are equal for every $i\neq j \in \{ 1, \dots , n\}$
for which the leading terms of $g_i$ and $g_j$ involve the same basis element $e_\ell$ of $M$. 
Then the set $\{ g_1, \dots g_n\}$ is a Gr\"{o}bner basis for $M$ if and only if
there exist, for all applicable $i,j \in \{ 1,\dots, n\}$, polynomials $h_{i,j,s}$ such that $f_{j,i}\cdot g_i - f_{i,j}\cdot g_j=0$ or
$$f_{j,i}\cdot g_i - f_{i,j} \cdot g_j = \sum_{s=1}^n h_{i,j,s} g_s$$
where the summands on the right hand side have leading terms less than or equal to
the leading term of $f_{j,i} \cdot g_i - f_{i,j} \cdot g_j$.
\end{theorem}

\qed

\vspace{5mm}

To apply this to our situation, first define 
$$\Psi_{\mathcal{A},\mathbb{A}_F^n} = R_{\mathcal{A}, \mathbb{A}_F^n} \otimes {\Lambda}_F [dz_1,\dots, dz_m]/ (dL \mid L \in \mathcal{Z}).$$
Note that this is a graded $R_{\mathcal{A}, \mathbb{A}_F^n}$-module by Grassmannian degree. In other words,
it can be expressed as the direct sum of the free $R_{\mathcal{A}, \mathbb{A}_F^n}$-modules 
\beg{GradedPieceOfPsiOverDLS}{M_r = {\Lambda}^r_{R_{\mathcal{A},\mathbb{A}_F^n}} [dz_1,\dots, dz_{m}]/
({\Lambda}^{r-1}_{R_{\mathcal{A},\mathbb{A}_F^n}} [dz_1,\dots, dz_m] \wedge F\{ dL \mid L \in \mathcal{Z}\}).}
Now, without loss of generality,
$z_1,\dots, z_{m_0}$ is a basis of the $F$-vector space generated by
$z_1,\dots, z_m$, for some $1\leq m_0 \leq m$. 
For any 
$$I=\{ i_1 < \dots < i_r\} \subseteq \{1,\dots, m\},$$
we can write
$$dz_I = dz_{i_1} \wedge \dots \wedge dz_{i_r} = a_1 dz_{I_1'} + \dots +a_k dz_{I_k'}$$
for some $I_1', \dots , I_k' \subseteq \{1,\dots, m_0\}$ of cardinality $r$, 
which are unique if we insist the sets $I_j'$ be different and the coefficients be non-zero.
In other words, the ${m_0 \choose r}$ elements $dz_{I'}$ for subsets $I' \subseteq \{1, \dots , m_0\}$
of cardinality $r$ form a basis of the free $R_{\mathcal{A},\mathbb{A}_F^n}$-module $M_r$.
We can write
$$M_r = {\Lambda}^r_{R_{\mathcal{A},\mathbb{A}_F^n}} [dz_1,\dots, dz_{m_0}].$$
(We write $=$ instead of $\cong$ to indicate that the isomorphism is canonical.)
We shall also write 
$$e_I = dz_I$$
for $I\subseteq \{ 1, \dots, m\}$.

Now $T_{\mathcal{A},\mathbb{A}_F^n}$ is the graded $R_{\mathcal{A},\mathbb{A}_F^n}$-submodule of $\Psi_{\mathcal{A},\mathbb{A}_F^n}$
whose degree $r$ submodule is generated by the elements
$$u_{i_1} \wedge \dots \wedge u_{i_r} = t_{i_1} \cdot \dots \cdot t_{i_r} dz_{i_1} \wedge \dots \wedge dz_{i_r}$$
with $I=\{ i_1 <\dots < i_r\} \subseteq \{ 1, \dots, m\}$.
Denote
$$t_I = t_{i_1} \cdot \dots \cdot t_{i_r}.$$

Now consider the $F[t_1 ,\dots, t_m]$-module
$$\Phi_r = F[t_1, \dots , t_m] \otimes {\Lambda}^r_F [dz_1,\dots , dz_{m_0}].$$
Then consider the submodules of $\Phi_r$ defined by
$$P_r = \langle t_I \cdot dz_I \mid |I|=r,\; I\subseteq \{1, \dots, m\}\rangle$$
$$N_r = \langle P_L \cdot dz_{I'} \mid |I'| =r,\; I' \subseteq \{1,\dots, m_0\}, L \in \mathcal{Z} \rangle.$$
We have (by the Homomorphism Theorem) 
$$T_{\mathcal{A}, \mathbb{A}_F^n} = \bigoplus_{r=1}^{m_0} (P_r/N_r\cap P_r).$$
Hence, it suffices to calculate the submodule of relations $N_r\cap P_r$ for a given $r$ and verify that it is generated by the 
${\bigwedge}_F[u_1, \dots , u_m]$-multiples of the $P_{L,S}$'s contained in $M_r$.

We will use the Buchberger criterion (Theorem \ref{eisenbudthm}) above to calculate $N_r \cap P_r$.
In general, for two submodules 
$\langle f_1,\dots, f_n\rangle$, $\langle g_1,\dots, g_m \rangle$ of a free module over a polynomial ring,
their intersection can be calculated by introducing another polynomial variable $s$ greater than the other variables
in the lexicographic order; it is then generated by all elements of a 
Gr\"{o}bner basis of
$$\langle f_1 \cdot s, \dots, f_n \cdot s, g_1 \cdot (1-s), \dots, g_m \cdot (1-s)\rangle$$ 
that do not contain $s$.

Letting 
$$L=a_1z_{i_1}+\dots+a_k z_{i_k}$$
where $1\leq i_1<\dots<i_k\leq m$, $a_i\neq 0\in F$, put 
\beg{eee1}{|L|:=\{i_1,\dots,i_k\}.}

Then Theorem \ref{t30} follows from the following

\begin{lemma}\label{lemmafort30grobner}
The set $\mathcal{S}_1 \cup \mathcal{S}_2 \cup \mathcal{S}_3$, where
\beg{Groebnerthreetermbasis}{\begin{array}{c}
\mathcal{S}_1 =\{ s\cdot t_{I_1\cup \dots \cup I_k}\cdot (a_1e_{I_1} +\dots +a_k e_{I_k})\mid\\ 
a_i \in F,\; I_1, \dots I_k \subseteq \{ 1, \dots m\}, \forall i \; |I_i|=r\},\\
\\ 
\mathcal{S}_2 =\{ (1-s)\cdot P_L\cdot e_{I'}\mid L\in \mathcal{Z}, \; I' \subseteq \{1, \dots, m_0\}, \; |I'| =r\}, \\
\\
\mathcal{S}_3 = \{ P_L \cdot t_{(I_1\cup\dots \cup I_k)\smallsetminus |L|} \cdot (a_1 e_{I_1} +\dots + a_k e_{I_k})\mid \\ 
a_i \in F, \; L\in \mathcal{Z}, \;I_1, \dots I_k \subseteq \{ 1, \dots m\}, \forall i \; |I_i|=r\}.
\end{array}}
forms a Gr\"{o}bner basis of the $F[s, t_1, \dots ,t_m]$-submodule of $\Phi_r \otimes_F F[s]$ generated by 
\beg{OriginalModuleBasisss}{
\begin{array}{c}
\{s\cdot t_I\cdot e_{I}\mid I\subseteq \{1, \dots, m\}, \; |I|=r , \\
(1-s)\cdot P_L\cdot e_{I'}\mid I' \subseteq \{ 1,\dots, m_0\}, \; |I'|=r , \; L \in \mathcal{Z}\}
\end{array}
}
with respect to the TOP lexicographic order specified above.
\end{lemma}

\vspace{3mm}
\noindent {\bf Remark:} Note that in the ring $\Psi_{\mathcal{A}, \mathbb{A}_F^n}$, the elements of $\mathcal{S}_3$ are in the ideal
$(P_{L,S})$. Concretely, 
$$P_L t_{I \smallsetminus |L|} e_I = t_{I \smallsetminus |L|} P_{L,I}\in \Psi_{\mathcal{A}, \mathbb{A}_F^n}$$
(see \rref{PLSDefinition}).
\vspace{3mm}

\begin{proof}[Proof of Lemma \ref{lemmafort30grobner}]

First observe that all elements of \rref{Groebnerthreetermbasis} are generated by \rref{OriginalModuleBasisss}. Note that this is only nontrivial
for $\mathcal{S}_3$, in which case it was already observed in Section \ref{Statementoftheresults} (where we introduced $P_{L,S}$).
We will prove Lemma \ref{lemmafort30grobner} by verifying the assumptions of the Buchberger criterion for any applicable pair of elements from \rref{Groebnerthreetermbasis}. This gives six cases:

\vspace{3mm}

\noindent {\bf Case 1:} $\mathcal{S}_1$ vs. $\mathcal{S}_1$.

Suppose we have two elements of $\mathcal{S}_1$, i.e. two nonzero elements
$$s\cdot t_{I_1\cup \dots \cup I_k} \cdot (a_1e_{I_1} +\dots a_k e_{I_k})$$
$$s\cdot t_{J_1\cup\dots \cup J_{\ell}}\cdot (b_1e_{J_1} +\dots b_{\ell} e_{J_{\ell}})$$
whose leading terms involve the same basis element $e_{I'}$, $I' \subseteq \{ 1,\dots , m_0\}$.
Without loss of generality, $a_i, b_j \in F^{\times}$ and $e_{I_1}$, $e_{J_1}$ involve the basis element of the highest degree. 
So, we must multiply the two elements by
$b_1 t_{J_1 \cup \dots \cup J_{\ell} \smallsetminus I_1 \cup \dots \cup I_k}$ and $a_1 t_{I_1 \cup \dots \cup I_k \smallsetminus J_1 \cup \dots \cup J_\ell}$
to match their leading terms.
Then the difference is
$$\begin{array}{c}
s\cdot t_{I_1\cup \dots\cup I_k\cup J_1\cup\dots\cup J_{\ell}}\cdot ( b_1 \cdot (a_1 e_{I_1} + a_2 e_{I_2} + \dots + a_k e_{I_k}) -\\ 
a_1\cdot (b_1 e_{J_1} +b_2 e_{J_2} +\dots + b_{\ell} e_{J_{\ell}})),
\end{array}$$
which is still an element of $\mathcal{S}_1$.

\vspace{3mm}

\noindent {\bf Case 2:} $\mathcal{S}_2$ vs. $\mathcal{S}_2$.

Suppose we have two elements of $\mathcal{S}_2$, say
\beg{2S2firstorigelt}{(1-s) \cdot P_L \cdot e_{I'}}
\beg{2S2secondorigelt}{(1-s)\cdot P_M\cdot e_{J'}}
for $I',J' \subseteq \{1,\dots, m_0\}$, $|I'| = |J'|= r$.
The Buchberger criterion gives a condition only when $e_{I'} = e_{J'}$, i.e. $I'=J'$. 
To match the leading terms of these two elements, we therefore multiply \rref{2S2firstorigelt}, \rref{2S2secondorigelt}
by monomials $a \cdot t_K$, $b\cdot t_{K'}$.
Then $LT(a \cdot t_K \cdot P_L) = LT( b\cdot t_{K'} \cdot P_M)$. 
However, by the results of Proudfoot and Speyer (Theorem 2 of \cite{foot}), the $P_L$'s form a universal Gr\"{o}bner basis.
Thus,
$$a \cdot t_K \cdot P_L -b\cdot t_{K'} \cdot P_M =\sum_N  t_{I_N}\cdot P_{L_N},$$
for some $t_{I_N}$, $P_{L_N}$ where the summands have lesser or equal leading terms than the left hand side, and hence,
$$a \cdot t_K \cdot (1-s) \cdot P_L \cdot e_{I'} - b\cdot t_{K'} \cdot (1-s)\cdot P_M\cdot e_{J'} =(1-s)\cdot e_{I'}\cdot (\sum_N P_{L_N} t_{I_N}),$$
which is a linear combination of elements of $\mathcal{S}_2$.

\vspace{3mm}

\noindent {\bf Case 3:} $\mathcal{S}_3$ vs. $\mathcal{S}_3$.

Suppose we have two different nonzero elements of $\mathcal{S}_3$, say
\beg{2S3bothorigelts}{\begin{array}{c}
P_L \cdot t_{(I_1 \cup \dots \cup I_k) \smallsetminus |L|} \cdot (a_1 \cdot e_{I_1} + \dots + a_k \cdot e_{I_k})\\
P_{L'} \cdot t_{(J_1 \cup \dots \cup J_{\ell}) \smallsetminus |L'|} \cdot (b_1\cdot e_{J_1} + \dots + b_{\ell} \cdot e_{J_{\ell}}).
\end{array}}
Again, the condition of the Buchberger criterion only applies when the leading terms of 
\beg{theleadingtermsareequal}{
\begin{array}{c}
a_1 \cdot e_{I_1} + \dots + a_k \cdot e_{I_k}\\
b_1\cdot e_{J_1} + \dots + b_{\ell} \cdot e_{J_{\ell}}
\end{array}
}
are equal up to non-zero scalar multiple, which, without loss of generality, we can assume to be equal to $1$.

First of all, note that without loss of generality, we have
\beg{WOLOGconditionsS3s}{
\begin{array}{c}
min (|L|) \notin I_1 \cup \dots \cup I_k\\
min (|L'|) \notin J_1 \cup \dots \cup J_{\ell},
\end{array}
}
since otherwise the elements \rref{2S3bothorigelts} are $t_{min (|L|)}$- resp. $t_{min (|L'|)}$-monomial
multiples of other elements of $\mathcal{S}_3$ which satisfy \rref{WOLOGconditionsS3s}
by applying the relations $dL$, respectively $dL'$, to \rref{theleadingtermsareequal}, using \rref{GradedPieceOfPsiOverDLS}.
(Note that this is where our definition of $P_{L,S}$ in Section \ref{Statementoftheresults} comes from.)

To simplify notation, from now on, we shall write
$$I= I_1 \cup \dots \cup I_k$$
$$J= J_1 \cup \dots \cup J_{\ell}$$
and abbreviate the elements \rref{theleadingtermsareequal} as $e_{(I)}$, $e_{(J)}$.
By assumption, we have 
\beg{LeadingtermsS3es}{LT(e_{(I)}) = LT (e_{(J)}).}

Our strategy will be to verify the condition of the Buchberger criterion separately on the pairs
of elements 
\beg{2S3firststepelts}{P_{L'}\cdot t_{I\smallsetminus |L'|}\cdot e_{(I)}, P_L \cdot t_{I\smallsetminus |L|} \cdot e_{(I)}}
and 
\beg{2S3secondstepelts}{P_{L'}\cdot t_{I\smallsetminus |L'|}\cdot e_{(I)} , P_{L'} \cdot t_{J\smallsetminus |L'|} \cdot e_{(J)}.}

First let us verify that this is valid. The elements \rref{2S3bothorigelts} can be brought to a minimal common leading term by
multiplying by $t_i$-monomials. Denote those elements with a common leading term by 
$$A = \overline{P_L \cdot t_{I\smallsetminus |L|}} \cdot e_{(I)}$$
$$C = \overline{P_{L'} \cdot t_{J\smallsetminus |L'|}} \cdot e_{(J)}.$$
First we need to show that
\beg{2S3Validitypreliminarymustdivide}{LT(P_{L'} \cdot t_{I\smallsetminus |L'|}\cdot e_{(I)}) \mid LT(A).}
To this end, note that the $t_i$-monomial factor of $LT(A) = LT(C)$ is of the form $t_H$ with
$$\begin{array}{c}
H = ((|L| \smallsetminus \{ min ( |L|)\}) \cup (I\smallsetminus |L|) ) \cup\\
((|L'| \smallsetminus \{ min (|L'|)\} )\cup (J\smallsetminus |L'|))
\end{array}$$
while $P_{L'}\cdot t_{I\smallsetminus |L'|}$ has leading term (up to scalar multiple)
$t_G$ with
$$G = (|L'| \smallsetminus \{ min (|L'|)\} ) \cup (I\smallsetminus |L'|).$$
Thus, we need to show $G\subseteq H$.

To this end, note that $H$ is $I\cup J \cup |L|\cup |L'|$ with possible exclusion of the elements $min(|L|)$, $min(|L'|)$,
while $G = (|L'| \cup I ) \smallsetminus \{ min (|L'|)\}$.
Thus, the only possible element of $G\smallsetminus H$ could be $min (|L|)$.
But by \rref{WOLOGconditionsS3s}, $min (|L|) \notin I$, so in that case
$$min (|L|) \in |L'| \smallsetminus \{ min (|L'|)\} \subseteq H.$$
Contradiction.

Thus, \rref{2S3Validitypreliminarymustdivide} is proved. However, this also proves
$$LT(P_{L'} \cdot t_{I\smallsetminus |L'|}) \mid LT(\overline{P_{L'} \cdot t_{J\smallsetminus |L'|}})$$
and hence
$$P_{L'} \cdot t_{I\smallsetminus |L'|} \mid  \overline{P_{L'} \cdot t_{J\smallsetminus |L'|}},$$
since both sides are equal up to a $t_i$-monomial multiple.
Put
$$B= \overline{P_{L'} \cdot t_{J\smallsetminus |L'|}} \cdot e_{(I)}.$$
Thus,
$$LT(A) = LT(B) =LT(C).$$
To verify that the pairs \rref{2S3firststepelts}, \rref{2S3secondstepelts} can be considered separately, 
by Theorem \ref{eisenbudthm}, it therefore suffices to show that
\beg{ConditionOnAnglelts}{LT(A- B), \; LT(B-C) \leq LT (A-C).}
Clearly, it now suffices to assume $e_{(I)}\neq e_{(J)}$ (since otherwise $B=C$). Similarly, we can assume $A \neq B$.
Then we have
$$LT(A-B) = LT (\overline{P_L \cdot t_{I\smallsetminus |L|}} - \overline{P_{L'} \cdot t_{J\smallsetminus |L'|}}) \cdot LT(e_{(I)})$$
and
$$LT(A-C) = LT(\overline{P_L \cdot t_{I\smallsetminus |L|}}) \cdot LT(e_{(I)} - e_{(J)})$$
by the TOP ordering.
Now we have
$$LT (\overline{P_L \cdot t_{I\smallsetminus |L|}} -\overline{P_{L'} \cdot t_{J\smallsetminus |L'|}}) < LT (\overline{P_{L'} \cdot t_{J\smallsetminus |L'|}}) = 
LT (\overline{P_L \cdot t_{I\smallsetminus |L|}}),$$
and thus again, since we are using the TOP ordering, 
$$LT(A-B) = LT (\overline{P_L \cdot t_{I\smallsetminus |L|}} -\overline{P_L' \cdot t_{J\smallsetminus |L'|}})\cdot LT(e_{(I)}) <$$
$$LT(\overline{P_L \cdot t_{I\smallsetminus |L|}}) \cdot LT(e_{(I)} - e_{(J)}) = LT (A-C).$$

Also, 
$$LT(B-C) = LT (\overline{P_{L'} \cdot t_{J\smallsetminus |L'|}})\cdot LT(e_{(I)} -e_{(J)}),$$
which equals $LT(A-C)$.
Thus, \rref{ConditionOnAnglelts} is proved.

\vspace{3mm}

\noindent {\em Subcase 3a:} The pair \rref{2S3firststepelts}.

We must verify the condition of the Buchberger criterion for 
$P_L \cdot t_{I\smallsetminus |L|} \cdot e_{(I)}$ and $P_{L'}\cdot t_{I\smallsetminus |L'|}\cdot e_{(I)}$.
Multiplying by $t_i$-monomials to get equal leading terms and applying the Buchberger criterion to $P_L$, $P_{L'}$, 
since the $P_M$'s form a universal Gr\"{o}bner basis by \cite{foot}, the difference is, again,
a linear combination 
\beg{2S3differencefirststep}{e_{(I)} \cdot \sum t_{Q_{i,j}}\cdot P_{M_i}}
for different relations $M_i$ (not scalar multiples of one another)
where the summands have lower leading terms.
Note that several different monomials $t_{Q_{i,j}}$ can be multiplied by the same element $P_{M_i}$, which is why the second index is needed.
Additionally, we have
\beg{QvsILL'}{I\smallsetminus (|L| \cup |L'|) \subseteq Q_{i,j},}
since all the polynomials we started with were multiples of $t_{I\smallsetminus (|L| \cup |L'|)}$.

Now each $M_i$ will be some linear combination of $L$ and $L'$ and thus can be written as
\beg{MiAiLBiL'}{M_i = a_i\cdot L + b_i\cdot L'.}
To see this, take a set of variables $z_s$ modulo the linear relations $L$ and $L'$. By the results of Proudfoot and Speyer \cite{foot},
applying the Buchberger criterion to $P_L$ and $P_{L'}$ (thought of as elements of the polynomial ring on the $t_s$)
then gives a linear combination of some of the polynomials $P_{M_i}$. Those are necessarily of the form \rref{MiAiLBiL'}
since those are the only relations present. However, the resulting equality (of the form ``a monomial multiple of $P_L$ plus a monomial
multiple of $P_{L'}$ equals a linear combination of $P_{M_i}$") will remain valid when other linear relations among the
$z_s$ are added (since this will only enlarge the ideal in the polynomial ring on the generators $t_s$)
and hence, after being multiplied by an appropriate monomial, can be used for \rref{QvsILL'}. This proves \rref{MiAiLBiL'}.

We need to show
\beg{S3step1whatwewant}{|(|L| \cup |L'|)\smallsetminus (Q_{i,j} \cup M_i)| \leq 1.}
The reason why \rref{S3step1whatwewant} suffices is that we have linear relations between the basis elements $e_J$. 
Recall that those elements are exterior multiples of the elements $e_j$. Now if, say, $j \in |L|$, then the relation $dL$ expresses $e_j$
as a linear combination of other elements $e_j'$ by \rref{GradedPieceOfPsiOverDLS}.
Under the assumption \rref{S3step1whatwewant}, such elements are accompanied
by a $t_{j'}$ factor, and thus are linear combinations of the basis elements $\mathcal{S}_3$. The argument of $L'$ is analogous.

To prove \rref{S3step1whatwewant}, consider again the relation \rref{MiAiLBiL'}. 
Let us study the possible sets of indices $s$ such that the coefficient at $z_s$ of a given non-zero linear combination of the relations $L$, $L'$
is 0. Those are sets of such indices s on which the ratio of the coefficients of $L$ and $L'$ at $z_s$ is constant.
There are only finitely many such ratios for which this set of indices is non-empty. I denote these ratios by
$q_k$ ($k\in K$ where $K$ is some finite set). Let us denote the set of indicies where the ratio is $q_k$ by $S_{q_k}$.

More precisely, let $L = \sum \alpha_s \cdot z_s$, $L' = \sum \beta_s \cdot z_s$,
$(\alpha_s, \beta_s)\neq (0,0)$ for $s\in |L|\cup |L'|$, and the disjoint sets
$$S_{q_k} = \{ s\in |L| \cup |L'| \mid [\alpha_s : \beta_s] = q_k\}$$
for different ratios $q_k$.
The possible $|M_i|$'s are of the form $|L|\cup |L'| \smallsetminus S_{q_k}$ (which includes $|L|$, $|L'|$, and $|L|\cup |L'|$).
Each $S_{q_k}\neq \emptyset $ is associated with at most one $M_i$ in \rref{MiAiLBiL'}
such that $|M_i|= |L| \cup |L'| \smallsetminus S_{q_k}$ (since we already know the ratio $[a_i : b_i]$ and we chose the $M_i$
not to be multiples of each other).
Thus, we have proven \rref{S3step1whatwewant} if we can rule out
\beg{S3step1whatcouldhappen}{|S_{q_k} \smallsetminus Q_{i,j}| \geq 2.}
However, \rref{S3step1whatcouldhappen} is impossible since if it occurred, then the monomial terms of 
$t_{Q_{i,j}} \cdot P_{M_i}$
could not be multiples of any of the monomial terms of $t_{Q_{i',j'}} \cdot P_{M_{i'}}$ for $i' \neq i$, $t_{I\smallsetminus |L|} \cdot P_L$,
or $t_{I\smallsetminus |L'|}\cdot P_{L'}$ each of which miss at most one variable $t_{\ell}$ for $\ell \in Q_{i,j}$.
Selecting a minimal $Q_{i,j}$ for a given $i$ (with respect to inclusion), 
the monomial terms of $t_{Q_{i,j}} \cdot P_{M_i}$ also cannot be multiples of any monomial terms of $t_{Q_{i,j'}}\cdot P_{M_i}$ for $j'\neq j$.
This contradicts the assumption that \rref{2S3differencefirststep} was obtained by applying the Buchberger criterion to the elements indicated.
This concludes the case of \rref{2S3firststepelts}.

\vspace{3mm}

\noindent {\em Subcase 3b:} The pair \rref{2S3secondstepelts}. 

We must verify the condition of the Buchberger criterion for
$$P_{L'} \cdot t_{I_1 \cup \dots \cup I_k \smallsetminus |L'|} \cdot (a_1e_{I_1} +\dots a_k e_{I_k}),$$
$$P_{L'} \cdot t_{J_1 \cup \dots \cup J_{\ell} \smallsetminus |L'|} \cdot (b_1e_{J_1} +\dots b_{\ell} e_{J_{\ell}}).$$
The Buchberger algorithm step gives
$$P_{L'} \cdot t_{I_1 \cup \dots \cup I_k \cup J_1 \cup \dots \cup J_{\ell} \smallsetminus |L'|} \cdot (a_1 e_{I_1} + \dots a_k e_{I_k} - b_1 e_{J_1} - \dots - b_{\ell} e_{J_{\ell}}),$$
which is an element of $\mathcal{S}_3$.

\vspace{5mm}

\noindent {\bf Case 4:} $\mathcal{S}_1$ vs. $\mathcal{S}_2$.

Now we will show that the Buchberger criterion's assumptions hold for an element of $\mathcal{S}_1$ and an element of $\mathcal{S}_2$.
Suppose we have two nonzero elements of the form
\beg{S1S2S1elt}{s\cdot t_{I_1 \cup \dots \cup I_k} \cdot (a_1 e_{I_1} + \dots + a_k e_{I_k}), \; I_i \subseteq \{1,\dots, m\}, \; |I_i| =r}
$$(1-s) \cdot P_L\cdot e_{I'}, \; I' \subseteq \{1,\dots, m_0\}, \; |I'|=r.$$
Then there exist some unique nonzero $b_1, \dots , b_{\ell}$'s and distinct $I_1', \dots , I_{\ell}'\subseteq \{ 1,\dots , m_0\}$ such that
$$a_1 \cdot e_{I_1} + \dots + a_k \cdot e_{I_k} = b_1 \cdot e_{I_1'} + \dots + b_{\ell} \cdot e_{I_{\ell}'}.$$
Without loss of generality, they are ordered $e_{I_1'} > \dots > e_{I_{\ell}'}$.
First, this implies, by definition,
$$P_L \cdot t_{(I_1 \cup \dots \cup I_k) \smallsetminus |L|}\cdot (b_1 e_{I_1'} + \dots + b_{\ell} e_{I_{\ell}'}) \in \mathcal{S}_3.$$
Take the minimal degree $t_i$-monomials $f$, $g$, such that the leading term of
\beg{S1S2multipliedbyf}{f\cdot (1-s) \cdot P_L\cdot e_{I'}}
equals the leading term of 
$$g\cdot s\cdot t_{I_1 \cup \dots \cup I_k} \cdot (b_1 e_{I_1'} + \dots + b_{\ell} e_{I_{\ell}'}).$$
Then we must have $I' = I_1'$. 
In addition,
$t_{(I_1 \cup \dots \cup I_k) \smallsetminus |L|}$ must also divide the monomial $f$.
Thus, for some $t_i$-monomial $h$, we have
$$LT(f\cdot (1-s)\cdot P_L\cdot e_{I'}) =$$ 
$$LT ( h \cdot (1-s) \cdot t_{(I_1 \cup \dots \cup I_k) \smallsetminus |L|} \cdot P_L \cdot
(b_1 e_{I_1'} + \dots + b_{\ell} e_{I_{\ell}'})),$$
and the terms including each $b_i e_{I_i'}$ are in $\mathcal{S}_2$.
Also, we have
$$LT( h\cdot (1-s) \cdot t_{(I_1 \cup \dots \cup I_k) \smallsetminus |L|} \cdot P_L\cdot (b_2 e_{I_2'} + \dots + b_{\ell} e_{I_{\ell}'}))$$
$$= LT (g\cdot s \cdot t_{I_1 \cup \dots \cup I_k}\cdot (b_2 e_{I_2'} + \dots + b_{\ell} e_{I_{\ell}'})).$$
So we can replace $(1-s) \cdot P_L \cdot e_{I'}$ by
\beg{S1S2replacementelt}{(1-s) \cdot P_L \cdot t_{I_1 \cup \dots \cup I_k \smallsetminus |L|} (b_1 e_{I_1'} + \dots + b_{\ell} e_{I_{\ell}'})}
in \rref{S1S2S1elt}.

Now in verifying the Buchberger criterion, if we expand the $P_L$ into a sum of $t_i$-monomials in the expression
\beg{S1S2expand}{-s \cdot t_{(I_1 \cup \dots \cup I_k) \smallsetminus |L|} \cdot P_L \cdot (b_1 e_{I_1'} + \dots + b_{\ell} e_{I_{\ell}'}),}
the term of the leading $t_i$-monomial times $g$ equals the leading term of \rref{S1S2multipliedbyf}, and the terms of \rref{S1S2expand}
of the other monomials
containing $t_{|L|\smallsetminus \{ i\}}$, $i\in L$ are in $\mathcal{S}_1$ by eliminating $i$
from $I_1 \cup \dots \cup I_k$ (if $i\in I_1 \cup \dots \cup I_k$) using the relation $L$, as above
in proving the sufficiency of the assumption \rref{WOLOGconditionsS3s}.
Therefore, the difference of \rref{S1S2S1elt} and \rref{S1S2replacementelt} is a sum of multiples of elements of $\mathcal{S}_1, \mathcal{S}_2, \mathcal{S}_3$
of lower or equal leading terms as required.

\vspace{3mm}

\noindent {\bf Case 5:} $\mathcal{S}_1$ vs. $\mathcal{S}_3$.

Suppose we have a nonzero element of $\mathcal{S}_1$ and a nonzero element of $\mathcal{S}_3$ of the forms
$$s \cdot t_{I_1 \cup \dots \cup I_k} \cdot (a_1 e_{I_1} + \dots a_k e_{I_k})$$
$$P_L \cdot t_{(J_1 \cup \dots \cup J_{\ell}) \smallsetminus |L|} \cdot (b_1 e_{J_1} + \dots b_{\ell} e_{J_\ell}).$$
Take the minimal degree $t_i$-monomials $f$, $g$, such that the leading terms of
$$f\cdot s \cdot t_{I_1 \cup \dots \cup I_k} \cdot (a_1 e_{I_1} + \dots a_k e_{I_k})$$
$$g \cdot P_L \cdot t_{(J_1 \cup \dots \cup J_{\ell}) \smallsetminus |L|} \cdot (b_1 e_{J_1} + \dots b_\ell e_{J_\ell})$$
are equal. Then $s$ must divide $g$.
Then, again, after expanding $P_L$ into a sum of $t_i$-monomials, in
$$g \cdot P_L \cdot t_{(J_1 \cup \dots \cup J_k) \smallsetminus |L|} \cdot (b_1 e_{J_1} + \dots b_\ell e_{J_\ell})
-f\cdot s \cdot t_{I_1 \cup \dots \cup I_k} \cdot (a_1 e_{I_1} + \dots a_k e_{I_k})$$
the term corresponding to the greatest monomial will cancel and the terms corresponding
to all other monomials
can be expressed as multiples of elements of $\mathcal{S}_1$ by again using the relation $L$, as above.

\vspace{3mm}

\noindent {\bf Case 6:} $\mathcal{S}_2$ vs. $\mathcal{S}_3$.

Finally, suppose we have an element of $\mathcal{S}_2$ and a nonzero element of $\mathcal{S}_3$ of the forms
$$(1-s) \cdot P_{L'} \cdot e_{I'}$$
$$P_L\cdot t_{(I_1\cup \dots \cup I_k)\smallsetminus |L|} \cdot (a_1 e_{I_1} + \dots + a_k e_{I_k})$$
Take the minimal degree $t_i$-monomials $f$, $g$, such that the leading terms of
\beg{S2S3fterm}{f \cdot (1-s) \cdot P_{L'} \cdot e_{I'}}
and
\beg{S2S3term}{g \cdot P_L \cdot t_{(I_1 \cup \dots \cup I_k) \smallsetminus |L|} \cdot (a_1 e_{I_1} + \dots a_k e_{I_k})}
are equal. Then $g= -s \cdot h$ for some $t_i$-monomial $h$.
If the difference \rref{S2S3fterm} and \rref{S2S3term} does not include $s$, then
$$f \cdot P_{L'} \cdot e_{I'} = h \cdot P_L \cdot t_{(I_1 \cup \dots \cup I_k) \smallsetminus |L|} \cdot (a_1 e_{I_1} + \dots a_k e_{I_k})$$
which is a $t_i$-monomial multiple of an element of $\mathcal{S}_3$.
Otherwise, we can replace \rref{S2S3term} by
\beg{S2S3replacement}{(1-s) \cdot h \cdot P_L \cdot t_{(I_1 \cup \dots \cup I_k) \smallsetminus |L|} \cdot (a_1 e_{I_1} + \dots a_k e_{I_k})}
because the additional term is a $t_i$-monomial multiple of an element of $\mathcal{S}_3$
(which in particular does not involve $s$, thus having a lower leading term than
the difference of \rref{S2S3fterm} and \rref{S2S3term}).
Now the difference of \rref{S2S3fterm} and \rref{S2S3replacement} is a sum of $t_i$-monomial multiples of elements of $\mathcal{S}_2$.
This concludes the proof of Lemma \ref{lemmafort30grobner}.

\end{proof}

\vspace{5mm}

Call $L$ {\em minimal} if there do not exist relations $L_1,L_2$ such that
$$L_1+L_2=L$$ $$|L_1|,|L_2|\subsetneq|L|.$$
Define shuffle permutations as follows:
for sets of natural numbers 
$S_1=\{i_1<\dots<i_k\}$, 
$ S_2=\{j_1<\dots<j_l\}$, 
$S_1\cap S_2=\emptyset,$ denote by $\sigma_{S_1,S_2}$ 
the permutation which puts the sequence 
$(i_1,\dots,i_k,j_1,\dots,j_l)$ in increasing order. Also define for $S=\{i_1<\dots<i_k\}$:
$$t_S:=t_{i_1}\dots t_{i_k}$$
$$u_S:=u_{i_1}\dots u_{i_k}$$
$$dz_S:=dz_{i_1}\dots dz_{i_k}$$

\begin{theorem} The ideals $I$ and $K$ are generated as follows:
$$I=(P_L| L \; is\; a \; minimal \; relation)$$
$$K=(P_{L,S}|L\; is\; a\; minimal\; relation\; and\; S\subseteq |L|)$$
\label{et6t}
(For the case of $I$, see \cite{foot}, Theorem 4.)
\end{theorem}

Let $L$ be as in \rref{a1}.
Recalling the notation of \rref{eee1}, let $ S\subseteq |L|$. Put $$ Q_{L,S}:=t_{|L|}dLdz_S.$$
So obviously, $Q_{L,S}\in K$.

\begin{lemma} \label{ll1}
$Q_{L,S}\in (P_{L,T}|T\subseteq|L|)$ \end{lemma} 
\begin{proof}
If $S\neq\emptyset$, let $i\in S$. 
Then $i \in S \subseteq |L|$, and hence,
$t_i | t_{|L|}$ and $dz_i | dz_S$.
Thus, $u_i = t_i dz_i$ must divide $t_{|L|} dL dz_S = Q_{L,S}$.
Then $Q_{L,S}=u_iP_{L,S}$.
On the other hand, by applying \rref{PLSDefinition}, we have, by definition,
$$u_{i_1} P_L -t_{i_1} P_{L, \{ i_1\}} = u_{i_1} P_L - t_i P_L dz_i + t_{i_1}  \dots  t_{i_\ell}dL= t_{|L|} dL = 
Q_{L,{\emptyset}}.
$$
In the first summand the surviving term is the term of
$P_{L,\emptyset}$ which omits $t_{i_1}$. 
In the second summand the surviving terms are the ``error terms" of the summand of $P_{L}$ which omits $t_{i_1}$. All 
remaining terms cancel. 
\end{proof}

\begin{proof}[Proof of Theorem \ref{et6t}:]
Even case :
Suppose we know $$L_1+L_2=L$$
$$|L_1|,|L_2|\subsetneq |L|.$$
Then $$P_L=t_{L\setminus {L_1}}P_{L_1}+t_{L\setminus {L_2}}P_{L_2}.$$   

Odd case :
If L is not minimal we know $L_1+L_2=L$ and $|L_1|,|L_2|\subsetneq |L|$. Based on the even case, the first guess for $P_{L,S}$ could be
\[\begin{split} &P_{L,S}':= P_{{L_1},{S_1}}u_{S\setminus{S_1}}t_{|L|\setminus{(|{L_1}|{\cup S})}}sign(\sigma_{S,S\setminus S_1})
\\&+P_{{L_2},{S_2}}u_{S\setminus S_2} t_{|L|\setminus(|L_2|\cup S)}sign(\sigma_{S_2,S\setminus S_2}),\end{split}\]
for 
$$S_1=S\cap |L_1|,\; S_2=S\cap |L_2|.$$
The terms that match are those when we omit $t_i$
from $t_L$ with $i\in |L|\setminus S$ or $i\in |L_1|\cap|L_2|\cap S$. The terms which do not match are 
for $i\in (|L_1|\cap S)\setminus|L_2|$ or $(|L_2|\cap S)\setminus|L_1|$. For $i\in (|L_1|\cap S)\setminus|L_2|$, the term 
missing in our first guess is 
$$q_i:=u_{S\setminus S_2\setminus\{i\}}Q_{{L_2},{S_2}}t_{|L|\setminus(|{L_2}|\cup 
S)}sign(\sigma_{S\setminus S_2\setminus\{i\},\{i\}})sign(\sigma_{ S\setminus S_2,S_2}).$$ 
Symmetrically, denote the missing term by $r_j$ for $j \in |L_2|\cap S\setminus|L_1|$. 
Thus, we have
$$P_{L,S} = P_{L,S}' + \sum_{i \in S\smallsetminus S_2} q_i + \sum_{j \in S\smallsetminus S_1} r_j.$$
Use Lemma \ref{ll1}.
\end{proof}

\section{The geometric interpretation}\label{s5}

Since the well known paper by W. Fulton and R. MacPherson \cite{ful}, compactifications of configuration spaces, and 
complements of hyperplane arrangements \cite{DC}, became an important topic of algebraic geometry. For a good survey, see
\cite{den}. Our geometric interpretation is related to a compactification known as the {\em reciprocal plane} \cite{den},
Section 5.1, and its super analog.

Let us assume the $z_j$'s linearly span the vector space $\mathbb{A}^n_F$ (otherwise, we can replace $x_1,\dots, x_n$ by a basis of
the span of $z_1,\dots,z_m$).
Denote 
$$\calA=\{z_1,\dots,z_m\}, \calA_S=\{z_i|i\in S\}.$$
Let $R_{\calA,\mathbb{A}_F^n}=F[t_1,\dots,t_m]/I$
(see Theorem \ref{t1}). 
We can then similarly write $R_{\mathcal{A},W}$ where $\mathcal{A}$ is a set of vectors spanning the dual of an $F$-vector space $W$.
A stratification of $\text{Spec} (R_{\calA,\mathbb{A}^n_F})$ can be described as follows.
Recall that we have a canonical embedding
\beg{eq0}{\mathbb{A}^n_F\setminus Z(z_1\dots z_m)\subseteq \text{Spec}(R_{\calA,\mathbb{A}^n_F}).}
Call a vector subspace $V\subseteq\mathbb{A}_F^n$ {\em special} if 
$V=Z(\calA_S)$ for some $S\subseteq\{1,\dots,m\}$. (Note: $S$ can be empty.) Put also
$$S_V=\{i\in \{1,\dots,m\}|V\subseteq Z(z_i)\}.$$
(Note \cite{den} that the sets of $i$'s for which the $z_i$'s are linearly independent are the independent sets of a matroid.
Then the sets $S_V$ are precisely what is called the {\em flats} of this matroid.) 
For a scheme $X$, denote by $|X|$ the underlying topological space. 
\vspace{3mm}

\begin{theorem} (\cite{foot}, Remark 6)
\label{tt0}
For $V\subseteq\mathbb{A}_F^n$ special, there is a canonical embedding
\beg{eq1}{\text{\em{Spec}}(R_{\calA_{S_V},\mathbb{A}_F^n/V}) \rightarrow \text{\em{Spec}}(R_{\calA,\mathbb{A}_F^n}). }

Composing (\ref{eq1}) with 
$$\mathbb{A}^n_F/V\setminus \bigcup_{i\in S} Z(z_i)\subseteq \text{\em{Spec}}(R_{\calA_{S_V},\mathbb{A}^n_F/V}),$$
(see (\ref{eq0})), induces a decomposition of sets (not topological spaces),

\beg{eq2}{|\text{\em{Spec}}(R_{\calA,\mathbb{A}_F^n})|=\coprod_{V\subseteq\mathbb{A}_F^n\text{ special}}|(\mathbb{A}_F^n/V)
\setminus\bigcup_{i\in S_V}Z(z_i)|.}
\end{theorem}

\begin{proof}
We have 
$$R_{\calA,\mathbb{A}_F^n}/(t_i|i\notin S_V)=R_{\calA_S,\mathbb{A}_F^n/V},$$
which gives the maps (\ref{eq1}). 
(The point is that there is no linear relation between the $z_i$'s in which all but one term would have $i\in S_V$.
Thus, all the relations $P_L$ where $L$ contains a term not in $S_V$ are in $(t_i|i\notin S_V)$.)

To prove \rref{eq2}, first note that the images of the inclusions of the components of the right hand side of \rref{eq2} are clearly disjoint since they correspond to imposing relations $t_i$ with $i\notin S_V$ for some special vector subspace $V$, and inverting all other $t_i$'s.
Thus, our task is to show that the canonical map from the right hand side to the left hand side of \rref{eq2} is onto.
To this end, let $Q\in \text{Spec}(R_{\mathcal{A},\mathbb{A}_F^n})$ and let
$$S=\{j\in\{1,\dots,m\}|Q\in (t_j)\}.$$
Let 
$$V=\bigcap_{j\in S}Z(z_j).$$
We want to prove that $S=S_V$.
The fact that $S\subseteq S_V$ is automatic. Suppose $j\in S_V\setminus S$.
Then $z_j=a_1z_{j_1}+\dots a_kz_{j_k}$ with $j_1<\dots<j_k\in S,$ $a_1,\dots, a_k\neq 0\in F$.
Let 
$$L=z_j-a_1z_{j_1}-\dots - a_kz_{j_k}.$$
By assumption, $Q\in (t_j)$. But in $R_{\mathcal{A},\mathbb{A}_F^n}/(t_j)$,
$P_L$ is a non-zero multiple of 
$$t_{j_1}\cdot \dots \cdot t_{j_k}.$$
This implies $Q\in (t_{j_i})$ for some $i=1,\dots,k$. Contradiction.
\end{proof}

Theorem \ref{tt0} suggests that $\text{Spec}(R_{\mathcal{A},\mathbb{A}_F^n})$ should have a compactification where on the right hand
side of (\ref{eq2}) we replace each
$$(\mathbb{A}_F^n/V) \setminus \bigcup_{i\in S_V}Z(z_i)$$
with the corresponding affine space $(\mathbb{A}_F^ n/V)$.
In fact, there is such a compactification $X_{\mathbb{A}_F^n,\calA}$ and it can be described as the 
Zariski closure of the image of the embedding
\beg{eq3}{\mathbb{A}_F^n\setminus Z(z_1\dots z_m) \overset{(z_1,\dots,z_m)}
\longrightarrow \prod_{i=1}^m\mathbb{P}_F^1.}
In the terminology of \cite{den}, this is an example of what is called a {\em toric compactification}. It was also
studied, from a different point of view, in \cite{ab}.
Note that while (\ref{eq3}) resembles superficially the formula for the 
De Concini-Procesi wonderful compactification \cite{DC}, (\ref{eq3}) is in fact quite different. While the wonderful
compactification uses projections to (typically) higher-dimensional projective spaces, (\ref{eq3}) uses inclusions of the affine
coordinates $z_i$ into $\mathbb{P}_F^1$.

\vspace{3mm}
The projective variety $X_{\mathbb{A}_F^n,\calA}$
is covered by a system of affine open sets, closed under intersection,
$$U_{V,T}=\text{Spec}\prod_{j\in T}z_j^{-1} F[t_i,z_j|i\notin S_V,j\in S_V]/(\frac{P_L}{t_{S_V\cap|L|}})$$
where $V$ runs through special subspaces of $\mathbb{A}_F^n$, $L$ runs through all linear relations among the 
$z_i$'s, and $T$ is any subset of $S_V$.
The following fact follows from the definitions:
\begin{lemma}
We have
$$U_{V,T}\bigcap U_{V',T'}=U_{W,T\cup T'\cup(S_V-S_{V'})\cup(S_{V'}-S_V)}$$
where
$$V+V'\subseteq W=\bigcap_{i\in S_V\cap S_{V'}}Z(z_i)$$
so
$$S_V\bigcap S_{V'}=S_W.$$
\end{lemma}
\qed

It follows from Theorem \ref{tt0} that $|U_{V,T}|$ are open subsets covering 
$X_{\mathbb{A}^n_F,\calA}$.
To show the affine schemes $U_{V,T}$
are reduced (their coordinate rings 
have no nilpotent elements), we have the following
generalization of Theorem \ref{t1}:

\begin{theorem}
Let $V$ be a special subspace of $\mathbb{A}_F^n$.
The kernel of the homomorphism of rings 
$$F[t_i,z_j|i\notin S_V,j\in S_V]\rightarrow \prod_{i\notin S_V} z_i^{-1}F[z_1,\dots,z_m]/(\mathcal{Z}_V)$$
given by $t_i\mapsto z_i^{-1}$, where $\mathcal{Z}_V$ is the set of all linear relations among the $z_i$'s, $i\in S_V$, is 
$$(\frac{P_L}{t_{S_V\cap |L|}}).$$ 
\end{theorem}

\begin{proof}
Note that by the proof of Theorem \ref{tt0},
any linear relation among the $z_i$'s which involves a
$z_i$ for $i\notin S_V$ involves at least two of them.
Therefore, we can repeat the induction in Section \ref{s2} with
$\{1,\dots,m\}$ replaced by $\{1,\dots,m\}\setminus S_V$.
\end{proof}

We also have a similar analog of Theorem \ref{t30}:
\begin{theorem}
Let $V$ be a special subspace of $\mathbb{A}_F^n$.
The kernel of the homomorphism of rings
$$
\diagram 
F[t_i,z_j|i\notin S_V,j\in S_V]\otimes \Lambda[u_i,dz_j|i\notin S_V,j\in S_V] \dto \\
\prod_{i\notin S_V}z_i^{-1}F[z_1,\dots,z_m]\otimes \Lambda [dz_i, \dots, dz_m] /(\mathcal{Y}_V)
\enddiagram $$
given by $t_i \mapsto z_i^{-1}$, $u_i \mapsto z_i^{-1} dz_i$, where $\mathcal{Y}_V=\mathcal{Z}_V\cup \{dL|L\in \mathcal{Z}_V\}$, is 
$$ (\frac{P_{L, S}}{t_{S_V \cap |L|}} ) $$
where $L$ runs through the linear relations among the $z_i$'s and $S \subseteq |L|$. 
\end{theorem}

Accordingly, we have a superscheme
analog $\widetilde{X}_{\mathbb{A}_F^n,\calA}$ of $X_{\mathbb{A}_F^n,\calA}$.
Here by a superscheme, we mean a locally ringed space by
$\mathbb{Z}/2$-graded commutative
rings which is locally isomorphic to $\text{Spec}$ of a $\mathbb{Z}/2$-graded
commutative ring (see e.g. \cite{west}).
$\widetilde{X}_{\mathbb{A}_F^n,\calA}$ is covered by super-affine open subsets
$$\begin{array}{c}
\widetilde{U}_{V,T}=\text{Spec} \prod_{j\in T} z_j^{-1}F[t_i,z_j|i\notin S_V,j\in S_V]\\
\otimes \Lambda[u_i,dz_j||i\notin S_V,j\in S_V]/
(\frac{P_{L,S}}{t_{T\cap |L|}}). 
\end{array} $$

We clearly have
$$|\widetilde{U}_{V,T}|=|U_{V,T}|$$
and for $|U_{V',T'}|\subseteq|U_{V,T}|$, $\widetilde{U}_{V',T'}$ is a complement of the zero set of an
(even) principal ideal in $\widetilde{U}_{V,T}$. Therefore, $\widetilde{X}_{\mathbb{A}_F^n,\calA}$ can be defined 
as the colimit of the $\widetilde{U}_{V,T}$'s in the category of superschemes.

\end{document}